\documentclass[12pt]{amsart}
\usepackage{amsmath,amssymb,latexsym,cancel,rotating}

\numberwithin{equation}{section}

\newtheorem{theorem}{Theorem}[section]
\newtheorem{proposition}[theorem]{Proposition}

\newtheorem{corollary}[theorem]{Corollary}
\newtheorem{lemma}[theorem]{Lemma}

\theoremstyle{definition}

\newtheorem{remark}[theorem]{Remark}
\newtheorem{example}[theorem]{Example}
\newtheorem{definition}[theorem]{Definition}

\input xy
\xyoption{poly}
\xyoption{2cell}
\xyoption{all}

\def\dim{\text{dim}}

\def\soc{\text{Soc}}

\def\nd{{\noindent}}

\def\ZZ{\mathbb{Z}}

\def\Acal{\mathcal{A}}

\def\Fcal{\mathcal{F}}

\def\Tcal{\mathcal{T}}

\def\dim{\text{dim}}

\renewcommand{\eqref}[1]{{\rm (\ref{#1})}}

\begin{document}

\title[On homomorphisms from Ringel-Hall algebras to quantum cluster algebras]
{On homomorphisms from Ringel-Hall algebras to quantum cluster algebras}

\author{Xueqing Chen, Ming Ding and Fan Xu}
\address{Department of Mathematical and Computer Sciences,
 University of Wisconsin-Whitewater\\
800 W.Main Street, Whitewater,WI.53190.USA}
\address{School of Mathematical Sciences and LPMC,
Nankai University, Tianjin, P.R.China}
\email{m-ding04@mails.tsinghua.edu.cn (M.Ding)}
\address{Department of Mathematical Sciences\\
Tsinghua University\\
Beijing 100084, P.~R.~China} \email{fanxu@mail.tsinghua.edu.cn
(F.Xu)}


\thanks{Ming Ding was supported by NSF of China (No. 11301282) and Specialized Research Fund for the Doctoral Program of Higher Education (No. 20130031120004) and Fan Xu was supported by NSF of China (No. 11071133).}

\subjclass[2000]{Primary  16G20, 17B67; Secondary  17B35, 18E30}

\date{\today}

\keywords{Ringel-Hall algebra, quantum cluster algebra, cluster variable, bipartite graph}

\maketitle

\begin{abstract}
 In \cite{rupel3},the authors  defined algebra homomorphisms from the dual Ringel-Hall algebra of certain hereditary abelian category $\mathcal{A}$ to an appropriate $q$-polynomial algebra. In the case that $\mathcal{A}$ is the representation category of  an acyclic  quiver, we give an alternative proof by using the cluster multiplication formulas in \cite{DX}. Moreover, if the underlying graph of $Q$ is bipartite and the matrix  $B$ associated to the
quiver $Q$ is of full rank, we show that the image of the algebra homomorphisms is in the corresponding quantum cluster algebra.
\end{abstract}


\section{Background}
The Ringel-Hall algebra $\mathcal{H}(\mathcal{A})$ of a (small) finitary abelian category $\mathcal{A}$ was introduced by Ringel (\cite{Ringel1990}). When $\mathcal{A}$ is the category $\mathrm{Rep}_{\mathbb{F}_q}Q$ of finite dimensional representations for a simply-laced quiver $Q$ over a finite field $\mathbb{F}_q$, the Ringel-Hall algebra $\mathcal{H}(\mathcal{A})$ is isomorphic to the positive part $U_q(\mathfrak{n})$ of the corresponding quantum group $U_q(\mathfrak{g})$ (\cite{Ringel1990}).  Lusztig (\cite{Lusztig}) constructed the canonical basis of the quantum group $U_q(\mathfrak{n})$ under the context of Ringel-Hall algebras. In order to study the canonical basis algebraically and combinatorially, Berenstein and Zelevinsky (\cite{berzel}) defined quantum cluster algebras as a noncommutative analogue of cluster algebras
(see \cite{ca1}\cite{ca2}).  A quantum cluster algebra  is a subalgebra of a skew field of rational functions in $q$-commuting
variable and generated by a
set of generators called the \textit{cluster variables}.

A natural question is to study the relations between Ringel-Hall algebras and quantum cluster algebras. Geiss, Leclerc and Schr\"oer (\cite{GLS}) showed that quantum groups of type $A, D$ and $E$ have quantum cluster structures. Recently, Berenstein and Rupel \cite{rupel3} constructed algebra homomorphisms from Ringel-Hall algebras to quantum cluster algebras. We recall their main result.
Let $\mathcal{A}$ be a finitary hereditary abelian category and
 $\mathbf{i}=(i_1,\cdots,i_m)$ be a sequence of simple objects in $\mathcal{A}$. They showed that,
under certain co-finiteness conditions, the assignment $[V]^{*}\rightarrow X_{V,\mathbf{i}}$
defines a homomorphism of algebras
 	$$\Psi_{\mathbf{i}}: \mathcal{H}^{*}(\mathcal{A})\rightarrow P_{\mathbf{i}}$$
where $\mathcal{H}^{*}(\mathcal{A})$ is dual Ringel-Hall algebra and $X_{V,\mathbf{i}}$ is the quantum cluster $\mathbf{i}$-character of $V$ in an appropriate $q$-polynomial algebra $P_{\mathbf{i}}$. Moreover, for an appropriate $\mathbf{i}$, the image restricting to the composition algebra of $\mathcal{H}^{*}(\mathcal{A})$ is in the corresponding upper cluster algebra.

The aim of this note is to give an alternative proof of the above result when $\mathcal{A}$ is the representation category of  an acyclic  quiver. Different from \cite{rupel3}, a key ingredient of our proof is to apply the cluster multiplication formulas proved in \cite{DX} (see also Theorem \ref{alg-homo}). We show that if the underlying graph of $Q$ is bipartite (i.e, we can associate this graph an orientation such that every vertex is a sink or a source) and the matrix  $B$ associated to the
quiver $Q$ is of full rank, then the algebra $\mathcal{AH}_{|k|}(Q)$
generated by all quantum cluster characters is exactly the quantum
cluster algebra $\mathcal{A}_{|k|}(Q)$ (see Theorem \ref{theorem}). As a corollary,  the image of the algebra homomorphism is in the quantum cluster algebra $\mathcal{A}_{|k|}(Q)$ (see Corollary \ref{included}). We expect that the approach in this note can be extended to construct algebra homomorphisms from derived Hall algebras to quantum cluster algebras.

\section{Quantum cluster algebras and Caldero-Chapoton maps}
\subsection{Quantum cluster algebras} We briefly recall the
definition of quantum cluster algebras. Let $L$ be a lattice of rank
$m$ and $\Lambda:L\times L\to \ZZ$ a skew-symmetric bilinear form.
We will need a formal variable $q$ and consider the ring of integer
Laurent polynomials $\ZZ[q^{\pm1/2}]$.  Define the \textit{based
quantum torus} associated to the pair $(L,\Lambda)$ to be the
$\ZZ[q^{\pm1/2}]$-algebra $\mathcal{T}$ with a distinguished
$\ZZ[q^{\pm1/2}]$-basis $\{X^e: e\in L\}$ and the  multiplication
given by
\[X^eX^f=q^{\Lambda(e,f)/2}X^{e+f}.\]
It is easy to see  that $\Tcal$ is associative and the basis
elements satisfy the following relations:
\[X^eX^f=q^{\Lambda(e,f)}X^fX^e,\ X^0=1,\ (X^e)^{-1}=X^{-e}.\]  It is known that $\Tcal$ is an Ore domain, i.e.,   is contained in its
skew-field of fractions $\Fcal$.  The quantum cluster algebra
 will be defined as a
$\ZZ[q^{\pm1/2}]$-subalgebra of $\Fcal$.

A \textit{toric frame} in $\Fcal$ is a map $M: \ZZ^m\to \Fcal
\setminus \{0\}$ of the form \[M({\bf c})=\varphi(X^{\eta({\bf
c})})\] where $\varphi$ is an automorphism of $\Fcal$ and $\eta:
\ZZ^m\to L$ is an  isomorphism of lattices.  By the definition, the
elements $M({\bf c})$ form a $\ZZ[q^{\pm1/2}]$-basis of the based
quantum torus $\Tcal_M:=\varphi(\Tcal)$ and satisfy the following
relations:
\[M({\bf c})M({\bf d})=q^{\Lambda_M({\bf c},{\bf d})/2}M({\bf c}+{\bf d}),\
M({\bf c})M({\bf d})=q^{\Lambda_M({\bf c},{\bf d})}M({\bf d})M({\bf
c}),\]
\[ M({\bf 0})=1,\ M({\bf c})^{-1}=M(-{\bf c}),\]
where $\Lambda_M$ is the skew-symmetric bilinear form on $\ZZ^m$
obtained from the lattice isomorphism $\eta$.  Let $\Lambda_M$ also
denote the skew-symmetric $m\times m$ matrix defined by
$\lambda_{ij}=\Lambda_M(e_i,e_j)$ where $\{e_1, \ldots, e_m\}$ is
the standard basis of $\ZZ^m$.  Given a toric frame $M$, let
$X_i=M(e_i)$.  Then we have
$$\Tcal_M=\ZZ[q^{\pm1/2}]\langle X_1^{\pm 1}, \ldots,
X_m^{\pm1}:X_iX_j=q^{\lambda_{ij}}X_jX_i\rangle.$$

Let $\Lambda$ be an $m\times m$ skew-symmetric matrix and let
$\tilde{B}$ be an $m\times n(m>n)$ matrix, whose principal part
denoted by $B$. We call the pair $(\Lambda, \tilde{B})$
\textit{compatible} if $\tilde{B}^T\Lambda=(D|0)$ is an $n\times m$
matrix with $D=diag(d_1,\cdots,d_n)$ where $d_i\in \mathbb{N}$ for
$1\leq i\leq n$. The pair $(M,\tilde{B})$ is called a
\textit{quantum seed} if the pair $(\Lambda_M, \tilde{B})$ is
compatible.  Now we define the mutation of the quantum seed $(\Lambda_M, \tilde{B})$ in  direction $k$
 for $1\leq k\leq n$.

Define the $m\times m$ matrix $E=(e_{ij})$ by
\[e_{ij}=\begin{cases}
\delta_{ij} & \text{if $j\ne k$;}\\
-1 & \text{if $i=j=k$;}\\
max(0,-b_{ik}) & \text{if $i\ne j = k$.}
\end{cases}
\]
For $n,k\in\ZZ$, $k\ge0$, denote ${n\brack
k}_q=\frac{(q^n-q^{-n})\cdots(q^{n-r+1}-q^{-n+r-1})}{(q^r-q^{-r})\cdots(q-q^{-1})}$.
Let ${\bf c}=(c_1,\ldots,c_m)\in\ZZ^m$ with $c_{k}\geq 0$.  Define
the toric frame $M': \ZZ^m\to \Fcal \setminus \{0\}$ as follows:
\begin{equation}\label{eq:cl_exp}M'({\bf c})=\sum^{c_k}_{p=0} {c_k \brack p}_{q^{d_k/2}} M(E{\bf c}+p{\bf b}^k),\ \ M'({\bf -c})=M'({\bf c})^{-1}.\end{equation}
where the vector ${\bf b}^k\in\ZZ^m$ is the $k-$th column of
$\tilde{B}$.

Define $\tilde{B}'=(b^{'}_{ij})$ by
\[b^{'}_{ij}=\begin{cases}
-b_{ij} & \text{if $i=k$ or $j=k$;}\\
b_{ij}+sgn(b_{ik}[b_{ik}b_{kj}]_{+}) & \text{otherwise.}
\end{cases}
\]
 where $[b]_{+}=max(b,0)$.

 Then the quantum seed $(M',\tilde{B}')$ is called
to be the mutation of $(M,\tilde{B})$ in direction $k$. Quantum
seeds are mutation-equivalent if they can be obtained from each
other by a sequence of mutations. Let $\mathcal{C}=\{M'(e_i):
i\in[1,n]\}$ where $(M',\tilde{B}')$ is mutation-equivalent to
$(M,\tilde{B})$.  Let $\mathbb{ZP}$ be the ring of integral Laurent
polynomials in the (quasi-commuting) variables in
$\{q^{1/2},X_{n+1},\cdots,X_{m}\}$. The \textit{quantum cluster
algebra} $\Acal_q(\Lambda_M,\tilde{B})$ is the
$\mathbb{ZP}$-subalgebra of $\Fcal$ generated by $\mathcal{C}$.

\begin{proposition}(Mutation of cluster variables)\cite{berzel}\label{lem:cluster_mutation}
The toric frame $X'$ is determined by
\begin{align}
X_k'&=X (\sum_{1\leq i\leq m}[b_{ik}]_{+} e_i -e_k)+X (\sum_{1\leq j\leq m}[-b_{jk}]_{+} e_j -e_k),\nonumber\\
X_k'&=X_i ,\quad 1\leq i\leq m,\quad i\neq k.\nonumber
\end{align}
\end{proposition}

\begin{theorem}(Quantum Laurent phenomenon)\cite{berzel}\label{laurent}
The quantum cluster algebra $\Acal_q(\Lambda_M,\tilde{B})$ is a
subalgebra of $\Tcal_M$.
\end{theorem}

Set $\textbf{X}=\{X_1,\cdots,X_n\}$ and
$\textbf{X}_{k}=\textbf{X}-\{X_k\}\cup \{X'_k\}$  for any $k\in
[1,n]$. Denote by $\mathcal{U}(\Lambda_M,\tilde{B})$ is the
$\mathbb{ZP}$-subalgebra of $\Fcal$ given by
$$\mathcal{U}(\Lambda_M,\tilde{B})=\mathbb{ZP}[\textbf{X}^{\pm 1}]\cap \mathbb{ZP}[\textbf{X}_{1}^{\pm 1}]\cap\cdots\cap \mathbb{ZP}[\textbf{X}_{n}^{\pm 1}].$$
We call  $\mathcal{U}(\Lambda_M,\tilde{B})$ the \textit{quantum upper
cluster algebra}. The following result shows that the acyclicity
condition closes the gap between the upper bounds and the
corresponding quantum cluster algebras.

\begin{theorem}\cite{berzel}\label{upper}
If the principal matrix $B$ is acyclic, then
$\mathcal{U}(\Lambda_M,\tilde{B})=\Acal_q(\Lambda_M,\tilde{B}).$
\end{theorem}

\subsection{Quantum Caldero-Chapoton maps}
Let $k$ be a finite field with cardinality $|k|=q$ and $m\geq n$ be
two positive integers and $\widetilde{Q}$ an acyclic quiver
with vertex set $\{1,\ldots,m\}$. Denote the subset
$\{n+1,\dots,m\}$ by $C$. The full subquiver $Q$ on the vertices
$1,\ldots,n$ is called the \emph{principal part} of $\widetilde{Q}$.
For $1\leq i\leq m$, let $S_i$ be the $i$th simple module for
$k\widetilde{Q}.$

Let $\widetilde{B}$ be the $m\times n$ matrix associated to the
quiver $\widetilde{Q}$ whose entry in position $(i,j)$ given by
\[
b_{ij}=|\{\mathrm{arrows}\, i\longrightarrow
j\}|-|\{\mathrm{arrows}\, j\longrightarrow i\}|
\]
for $1\leq i\leq m$, $1\leq j\leq n$. Denote by $\widetilde{I}$ the
left $m\times n$ submatrix of the identity matrix of size $m\times
m$. Assume that there exists some antisymmetric $m\times m$ integer
matrix $\Lambda$ such that
\begin{align}\label{eq:simply_laced_compatible}
\Lambda(-\widetilde{B})=\begin{bmatrix}I_n\\0
\end{bmatrix},
\end{align}
where $I_n$ is the identity matrix of size $n\times n$. Let
$\widetilde{R}=\widetilde{R}_{\widetilde{Q}}$ be the $m\times n$
matrix with its entry in position $(i,j)$ given by
\[
\widetilde{r}_{ij}:=\mathrm{dim}_{k}\mathrm{Ext}^{1}_{k\widetilde{Q}}(S_j,S_i)=|\{\mathrm{arrows}\,
j\longrightarrow i\}|.
\] for $1\leq i\leq m$, $1\leq j\leq n$. Set
$\widetilde{R}^{tr}=\widetilde{R}_{\widetilde{Q}^{op}}.$
 Denote the principal $n\times n$ submatrices of
 $\widetilde{B}$ and $\widetilde{R}$ by $B$ and $R$
respectively. Note that
$\widetilde{B}=\widetilde{R}^{tr}-\widetilde{R}$ and $B=R^{tr}-R$.

Let $\mathcal C_{\widetilde{Q}}$ be the cluster category of $k
\widetilde{Q}$, i.e., the orbit category of the derived category
$\mathcal{D}^b(\widetilde{Q})$ under the action of  the functor
$F=\tau\circ[-1]$ (see \cite{BMRRT}). Let  $I_i$ be the
indecomposable injective $k \widetilde{Q}$ module for $1\leq i \leq
m.$ Then the indecomposable $k \widetilde{Q}$-modules and $I_i[-1]$
for $1\leq i \leq m$ exhaust all indecomposable objects of the
cluster category $\mathcal C_{\widetilde{Q}}$. Each object $M$ in
$\mathcal C_{\widetilde{Q}}$ can be uniquely decomposed as
$$M=M_0\oplus I_M[-1]$$
where $M_0$ is a module and $I_M$ is an injective  module.

The Euler form on $k \widetilde{Q}$-modules $M$ and $N$ is given by
$$\langle M,N\rangle=\mathrm{dim}_{k}\mathrm{Hom}(M,N)-\mathrm{dim}_{k}\mathrm{Ext}^{1}(M,N).$$
Note that the Euler form only depends on the dimension vectors of
$M$ and $N$.

The quantum Caldero-Chapoton map of an acyclic quiver $Q$ has been
defined in \cite{rupel}\cite{fanqin}\cite{DX}\cite{D}:
$$X_?: \mathrm{obj}\mathcal C_{\widetilde{Q}}\longrightarrow \Tcal$$
 by the following rules:\\
(1)\ If $M$ is a $k Q$-module, then
                    $$
                       X_{M}=\sum_{\underline{e}} |\mathrm{Gr}_{\underline{e}} M|q^{-\frac{1}{2}
\langle
\underline{e},\underline{m}-\underline{e}\rangle}X^{-\widetilde{B}\underline{e}-(\widetilde{I}-\widetilde{R}^{tr})\underline{m}};$$
(2)\ If $M$ is a $k Q$-module and $I$ is an injective $k
\widetilde{Q}$-module, then
                    $$
                       X_{M\oplus I[-1]}=\sum_{\underline{e}} |\mathrm{Gr}_{\underline{e}} M|q^{-\frac{1}{2}
\langle
\underline{e},\underline{m}-\underline{e}-\underline{i}\rangle}X^{-\widetilde{B}\underline{e}-(\widetilde{I}-\widetilde{R}^{tr})\underline{m}+\underline{\mathrm{dim}}
soc I},
                    $$
where $\underline{\mathrm{dim}} I= \underline{i},
\underline{\mathrm{dim}} M= \underline{m}$ and
$\mathrm{Gr}_{\underline{e}}M$ denotes the set of all submodules $V$
of $M$ with $\underline{\mathrm{dim}} V= \underline{e}$. We note
that
$$
X_{P[1]}=X_{\tau P}=X^{\underline{\mathrm{dim}} P/rad
P}=X^{\underline{\mathrm{dim}}\mathrm{soc}I}=X_{I[-1]}=X_{\tau^{-1}I}.
$$
for any projective $k\widetilde{Q}$-module $P$ and injective
$k\widetilde{Q}$-module $I$ with $\mathrm{soc}I=P/\mathrm{rad}P.$

In the following, we always use  the underlined
lower  letter $\underline{x}$ to denote  the corresponding  dimension vector of a
$kQ$-module $X$ and view $\underline{x}$ as a column vector in
$\mathbb{Z}^n.$

\section{The dual Ringel-Hall algebras and the cluster multiplication formulas}
Let $\mathcal{A}$ be the representation category of  an acyclic  quiver $Q$. For an object $V\in \mathcal{A}$, we will
write $[V]$ for the isomorphism class of V and write $|V|$ for the class of $V$ in the Grothendieck group
$K(\mathcal{A})$. Let $\mathcal{H}(\mathcal{A})=\bigoplus k[V]$ be the free $K(\mathcal{A})$-graded $k$-vector space spanned by the isomorphism
classes of objects of $\mathcal{A}$ with the natural grading via class in $K(\mathcal{A})$. For $U,V,W\in \mathcal{A}$ define
$$g_{UW}^{V}=|\{R\subset V|R\cong W,V/R\cong U\}|.$$
The assignment $[U][W]=\sum_{[V]}g_{UW}^{V}[V]$ defines an associative multiplication on $\mathcal{H}(\mathcal{A})$.
The algebra $\mathcal{H}(\mathcal{A})$ is known as the Ringel-Hall algebra. Denote by $\mathcal{H}^{*}(\mathcal{A})$
the  dual Ringel-Hall algebra, which is the space of linear functions $\mathcal{H}(\mathcal{A})\rightarrow k$ with a basis of
all delta-functions $\delta_{V}$ labeled by isomorphism classes $[V]$ of objects of $\mathcal{A}$.
\begin{proposition}\label{Hall-Ringel}
Let $M$ and $N$  be $kQ$-modules, then the assignment
$$\delta_{M}\ast\delta_{N}=q^{\frac{1}{2}\Lambda((\widetilde{I}-\widetilde{R}^{'})\underline{m},
(\widetilde{I}-\widetilde{R}^{'})\underline{n})+\langle\underline{m},\underline{n}\rangle}\sum_{E}h^{MN}_{E}\delta_E$$ defines an associative multiplication on $\mathcal{H}^{*}(\mathcal{A})$, where $h^{MN}_{E}=\frac{|\mathrm{Ext}_{kQ}^{1}(M,N)_{E}|}{|\mathrm{Hom}_{kQ}(M,N)|}$.
\end{proposition}
\begin{proof}
Note that $\sum_{E}g^{E}_{MN}g^{F}_{EL}=\sum_{G}g^{G}_{NL}g^{F}_{MG},$ and the relation of $h^{MN}_{E}$ and $g^{E}_{MN}$ is given by the Riedtmann-Peng's formula
$$h^{MN}_{E}=g^{E}_{MN}|Aut(M)||Aut(N)||Aut(E)|^{-1}.$$
Thus we have $\sum_{E}h^{E}_{MN}h^{F}_{EL}=\sum_{G}h^{G}_{NL}h^{F}_{MG}.$ It is easy to see that
$\phi(\underline{m},\underline{n}):=\frac{1}{2}\Lambda((\widetilde{I}-\widetilde{R}^{'})\underline{m},
(\widetilde{I}-\widetilde{R}^{'})\underline{n})+\langle\underline{m},\underline{n}\rangle$ is a bilinear form on $\mathbb{Z}^{n}$. Hence the associativity can be  deduced.
\end{proof}

For any $k\widetilde{Q}-$modules $M,N$ and $E$, denote by
$\varepsilon_{MN}^{E}$  the cardinality of the set
$\mathrm{Ext}_{k\widetilde{Q}}^{1}(M,N)_{E}$ which is the subset of
$ \mathrm{Ext}_{k\widetilde{Q}}^{1}(M,N)$ consisting of those
equivalence classes of short exact sequences with middle term
isomorphic to $E$.  Define
$$\mathrm{Hom}_{k\widetilde{Q}}(M,I)_{BI'}:=\{f:M\longrightarrow I|ker f\cong B,coker f\cong
I'\}.$$
Denote $$[M,N]=\mathrm{dim}_{k}\mathrm{Hom}_{k\widetilde{Q}}(M,N),$$
$$[M,N]^{1}=\mathrm{dim}_{k}\mathrm{Ext}_{k\widetilde{Q}}^{1}(M,N).$$
We have the
following cluster multiplication formulas.
\begin{theorem}\cite{DX}\cite{D}\label{multi}
Let $M$ and $N$  be any $kQ$-modules,  and $I$ any injective
$k\widetilde{Q}$-module, then
$$(1)\ q^{[M,N]^{1}}X_{M}X_{N}=q^{\frac{1}{2}\Lambda((\widetilde{I}-\widetilde{R}^{'})\underline{m},
(\widetilde{I}-\widetilde{R}^{'})\underline{n})}
\sum_{E}\varepsilon_{MN}^{E}X_E,$$
$$(2)\ q^{[M,I]}X_{M}X_{I[-1]}=q^{\frac{1}{2}\Lambda((\widetilde{I}-\widetilde{R}^{'})\underline{m},
-\mathrm{\underline{dim}}\soc I)}
\sum_{B,I'}|\mathrm{Hom}_{k\widetilde{Q}}(M,I)_{BI'}|X_{B\oplus
I'[-1]}.$$
\end{theorem}

   Note that Theorem \ref{multi}(1) implies the following result which has been proved by Berenstein-Rupel  using generalities on bialgebras in braided
monoidal categories.
\begin{theorem}\cite{rupel3}\label{alg-homo}
The assignment $\delta_{V}\rightarrow X_{V}$ defines an algebra homomorphism $\Psi: \mathcal{H}^{*}(\mathcal{A})\rightarrow \mathcal{T}$.
\end{theorem}
\nd \emph{An alternative proof}:
Note that the first cluster multiplication formula in Theorem \ref{multi} can be rewritten as
$$X_{M}X_{N}=q^{\frac{1}{2}\Lambda((\widetilde{I}-\widetilde{R}^{'})\underline{m},
(\widetilde{I}-\widetilde{R}^{'})\underline{n})+<\underline{m},\underline{n}>}\sum_{E}h^{MN}_{E}X_E.$$
Thus we have
\begin{eqnarray}
\nonumber \Psi(\delta_{M}\ast\delta_{N}) & = & \Psi(q^{\frac{1}{2}\Lambda((\widetilde{I}-\widetilde{R}^{'})\underline{m},
(\widetilde{I}-\widetilde{R}^{'})\underline{n})+\langle\underline{m},\underline{n}\rangle}\sum_{E}h^{MN}_{E}\delta_E)\\
\nonumber  & = & q^{\frac{1}{2}\Lambda((\widetilde{I}-\widetilde{R}^{'})\underline{m},
(\widetilde{I}-\widetilde{R}^{'})\underline{n})+\langle\underline{m},\underline{n}\rangle}\sum_{E}h^{MN}_{E}X_E\\
\nonumber & = & X_{M}X_{N}=\Psi(\delta_{M})\Psi(\delta_{N}).
\end{eqnarray}
This completes the proof.
\hfill$\square$

\section{Quantum cluster algebras  for bipartite graphs}
In this section, we assume that $Q$ is an acyclic  quiver whose
underlying graph is bipartite and the matrix  $B$ associated to the
quiver $Q$ is of full rank. Note that in this case the corresponding quantum cluster algebras are coefficient-free.
 We will show that the algebra
$\mathcal{AH}_{|k|}(Q)$ generated by all quantum cluster characters
is equal to the quantum cluster algebra $\mathcal{A}_{|k|}(Q)$.
\begin{definition}\label{def}
$X_{L}$ is called \emph{the quantum cluster character} if
$L\in\mathcal C_{Q}$.
\end{definition}

\begin{definition}\label{def2}
For a quiver $Q$, denote by $\mathcal{AH}_{|k|}(Q)$ the
 $\mathbb{Z}$-subalgebra of $\mathcal{F}$ generated by
all the quantum cluster characters.
\end{definition}

 Let $Q$ be an acyclic
quiver and $i$ be a sink or a source in $Q$. We define the reflected
quiver $\sigma_i(Q)$ by reversing all the arrows ending at $i$. An
\emph{admissible sequence of sinks (resp. sources)} is a sequence
$(i_1, \ldots, i_l)$ such that $i_1$ is a sink (resp. source) in $Q$
and $i_k$ is a sink (resp source) in $\sigma_{i_{k-1}}\cdots
\sigma_{i_1}(Q)$ for any $k=2, \ldots, l$. A quiver $Q'$ is called
\emph{reflection-equivalent}\index{reflection-equivalent} to $Q$ if
there exists an admissible sequence of sinks or sources $(i_1,
\ldots, i_l)$ such that $Q'=\sigma_{i_{l}}\cdots \sigma_{i_1}(Q)$.
Note that mutations can be viewed as generalizations of reflections,
i.e, if $i$ is a sink or a source in a quiver $Q$,
 then $\mu_i(Q)=\sigma_i(Q)$ where $\mu_i$ denotes the mutation in the direction
 $i$. We suppose that  $Q'$ is a quiver mutation-equivalent to $Q$.  Denote by
$\Phi_{i}: \mathcal{A}_{|k|}(Q)\rightarrow
\mathcal{A}_{|k|}(Q')$ the natural canonical isomorphism sending each initial cluster variable of $\mathcal{A}_{|k|}(Q)$ to its Laurent expansion in the initial cluster of $\mathcal{A}_{|k|}(Q')$.

Let $\Sigma_i^+:\ rep(Q) \longrightarrow \ rep(Q')$ be the standard
BGP-reflection functor and $R_i^+:\mathcal C_Q \longrightarrow
\mathcal C_{Q'}$ be the
        extended BGP-reflection functor defined in \cite{Zhu}:
        $$R_i^+:\left\{\begin{array}{rcll}
            X & \mapsto & \Sigma_i^+(X) & \textrm{ if }X \not \simeq S_i \textrm{ is a module}\\
            S_i & \mapsto & P_i[1] \\
            P_j[1] & \mapsto & P_j[1] & \textrm{ if }j \neq i\\
            P_i[1] & \mapsto & S_i
        \end{array}\right.$$
    In \cite{rupel}, the author proved the following result.
\begin{theorem}\cite{rupel}\label{ref}
For any indecomposable object $M$ in  $\mathcal C_{Q}$, we have
$\Phi_{i}(X_M^{Q})=X_{R_i^+M}^{Q'}.$
\end{theorem}
The following lemma is well-known.
\begin{lemma}\cite[Lemma 8(b)]{CK2005}\label{easy}
        Let $$M \longrightarrow E \longrightarrow N \longrightarrow M[1]$$
        be a non-split triangle in $\mathcal C_{Q}.$ Then
        $$\mathrm{dim}_{k}\mathrm{Ext}^{1}_{\mathcal C_{Q}}(E,E) < \mathrm{dim}_{k}\mathrm{Ext}^{1}_{\mathcal C_{Q}}(M \oplus N, M \oplus N).$$
\end{lemma}

\begin{theorem}\label{theorem}
Assume that  $Q$ is  an acyclic  quiver  whose underlying graph is bipartite
and the matrix  $B$ associated to the quiver $Q$  is  of  full rank, then
$\mathcal{AH}_{|k|}(Q)=\mathcal{A}_{|k|}(Q).$
\end{theorem}
\begin{proof}

Firstly, we prove that for any indecomposable object $M\in\mathcal
C_{Q}$, $X_M$ is in the quantum cluster algebra
$\mathcal{A}_{|k|}(Q)$.

{\em Case 1: If $Q$ is an alternating quiver (i.e, whose vertex is a sink or a source).}

Denoted by
$$\Phi_{i}: \mathcal{A}_{|k|}(Q)\rightarrow
\mathcal{A}_{|k|}(Q')$$ the canonical isomorphism of quantum cluster
algebras associated to sink or source $1\leq i\leq n$. It follows
from Theorem \ref{ref}, we obtain that
$\Phi_{i}(X_M^{Q})=X_{R_i^{\pm}{M}}^{Q'}$ for any indecomposable
object $M\in\mathcal C_{Q}$. It is easy to see that $Q'$ is again an
acyclic quiver. Then we obtain  that
$$X_M\in \mathbb{Z}[\textbf{X}^{\pm 1}]\cap \mathbb{Z}[\textbf{X}_{1}^{\pm 1}]\cap\cdots\cap \mathbb{Z}[\textbf{X}_{n}^{\pm 1}].$$
Note that the quiver  $Q$ is acyclic, thus the corresponding quantum
upper cluster algebra associated to $Q$  coincides with the quantum
cluster algebra $\mathcal{A}_{|k|}(Q)$ (see Theorem \ref{upper}).
Hence $X_M$ is in the quantum cluster algebra
$\mathcal{A}_{|k|}(Q)$.

{\em Case 2: If $Q$ is an acyclic  quiver whose underlying graph is
bipartite.}

Note that $Q$  is reflection equivalent to some
alternating quiver $Q'$ and in the {\em Case 1} we have showed that for any
indecomposable object $M\in\mathcal C_{Q'}$, $X_M$ is in the
corresponding quantum cluster algebra $\mathcal{A}_{|k|}(Q')$. Thus
the rest of the proof immediately follows from Theorem \ref{ref}.

Now we  need to prove that for any quantum cluster character
$X_{L}\in\mathcal{AH}_{|k|}(Q)$, then
$X_{L}\in\mathcal{A}_{|k|}(Q)$. Let $L\cong
\bigoplus_{i=1}^{l}L_{i}^{\oplus n_{i}}, n_{i}\in \mathbb{N}$ where
$L_{i}\ (1\leq i\leq l)$ are indecomposable objects in $\mathcal
C_{Q}$. According to Theorem \ref{multi}, we arrive at
the following equalty
$$X^{n_{1}}_{L_{1}}X^{n_{2}}_{L_{2}}\cdots X^{n_{l}}_{L_{l}}=q^{\frac{1}{2}n_{L}}X_{L}+
\sum_{\dim_{k}\mathrm{Ext}^{1}_{\mathcal
C_{Q}}(E,E)<\dim_{k}\mathrm{Ext}^{1}_{\mathcal
C_{Q}}(L,L)}f_{n_{E}}(q^{\pm\frac{1}{2}})X_E$$ where $n_{L}\in
\mathbb{Z}$ and $f_{n_{E}}(q^{\pm\frac{1}{2}})\in
\mathbb{Z}[q^{\pm\frac{1}{2}}].$ Using Lemma \ref{easy} and
proceeding  by induction, it is straightforward to verify that
$X_{L}\in \mathcal{A}_{|k|}(Q)$.
\end{proof}
By Theorem \ref{theorem}, we can deduce the following corollary.
\begin{corollary}\label{included}
Assume that  $Q$ is  an acyclic  quiver  whose underlying graph is bipartite,
and the matrix  $B$ associated to the quiver $Q$  is  of  full rank, then
$\Psi(\mathcal{H}^{*}(\mathcal{A}))\subseteq\mathcal{A}_{|k|}(Q).$
\end{corollary}
\begin{proof}
By Theorem \ref{alg-homo}, we have $\Psi(\mathcal{H}^{*}(\mathcal{A}))\subseteq\mathcal{AH}_{|k|}(Q).$ Hence the proof immediately follows from Theorem \ref{theorem}.
\end{proof}

\begin{remark}
It is natural to ask when $\Psi(\mathcal{H}^{*}(\mathcal{A}))$ is equal to $\mathcal{A}_{|k|}(Q).$ The key ingredient of this problem is to prove that the initial cluster variables can be written as a $\ZZ[q^{\pm1/2}]-$combination of some product of cluster characters associated to $kQ$-modules. In the following, we give an example in this direction.
\end{remark}

\begin{example}
 We set  $\Lambda=\left(\begin{array}{cc} 0 & 1\\ -1 &
0\end{array}\right)$ and $B=\left(\begin{array}{cc} 0 & 2\\
-2 & 0\end{array}\right)$. Thus the quiver $Q$ associated to this pair is
the Kronecker quiver: \vspace*{-0.5cm}
\begin{center}
 \setlength{\unitlength}{0.61cm}
 \begin{picture}(5,4)
 \put(0,2){1}\put(0.4,2.2){$\bullet$}
\put(3,2.2){$\bullet$}\put(3.4,2){2} \put(0.8,2.5){\vector(3,0){2}}
\put(0.8,2.2){\vector(3,0){2}}
 \end{picture}
 \end{center}
\vspace*{-1cm}

Let $k$ be a finite field with cardinality $|k|=q^2$. The category
$rep(kQ)$ of finite-dimensional representations can be identified
with the category of mod-$kQ$ of finite-dimensional modules over the
path algebra $kQ.$ It is well-known (see \cite{dlab}) that up to isomorphism the
indecomposable $kQ$-module contains  three
families: the preprojective
modules with dimension vector $(n-1,n)$ (denoted by $M(n)$), the indecomposable regular modules with dimension vector
$(nd_p, nd_p)$ for $p\in \mathbb{P}^1_k$ of degree $d_p$ (in
particular, denoted by $R_p(n)$ for $d_p=1$) and the
preinjective modules with dimension vector $(n,n-1)$ (denoted by
$N(n)$). Here $n\in \mathbb{N}$.

For $m\in \mathbb{Z}\setminus\{1, 2\}$, set
\[V(m)=\begin{cases}
N(m-2)& \text{if $m\geq 3$;}\\
M(-m+1) & \text{if $m\leq 0$.}
\end{cases}
\]

Now, let  $\Tcal=\ZZ[q^{\pm 1/2}]\langle X_1^{\pm1}, X_2^{\pm1}:
X_1X_2=qX_2X_1\rangle$ and  ${\mathcal F}$ be the skew field of
fractions of $\Tcal$. The quantum cluster algebra of the
Kronecker quiver is the $\ZZ[q^{\pm 1/2}]$-subalgebra of
${\mathcal F}$ generated by the cluster variables in $\{X_k |k\in\ZZ\}$
defined recursively by
$$X_{m-1}X_{m+1}= qX_m^2+1.$$
With the above notation, we have  the following results: 
\begin{lemma}\cite{rupel}\label{rank2}
For any   $m\in\ZZ\setminus \{1,2\}$, the $m$-th cluster variable
$X_m$ of $\Acal_q(2,2)$ equals $X_{V(m)}$.
\end{lemma}

\begin{lemma}\cite{DX1}\label{rank2-kro}
For any $n\in \mathbb{Z},$ we have
$$X_{n}X_{R_p(1)}=q^{-\frac{1}{2}}X_{n-1}+q^{\frac{1}{2}}X_{n+1}.$$
\end{lemma}

\begin{theorem}\label{included-kro}
Assume that  $Q$ is  the Kronecker quiver, then
$\Psi(\mathcal{H}^{*}(\mathcal{A}))=\mathcal{A}_{|k|}(Q).$
\end{theorem}
\begin{proof}
 By Corollary \ref{included}, we know that $\Psi(\mathcal{H}^{*}(\mathcal{A}))\subseteq\mathcal{A}_{|k|}(Q).$ Note that $\Psi$ is  an algebra homomorphism according to Theorem \ref{alg-homo}, thus it is enough to prove that $X_{1}$ and $X_{2}$ have preimages. By Lemma \ref{rank2-kro}, we have $X_{0}X_{R_p(1)}=q^{-\frac{1}{2}}X_{-1}+q^{\frac{1}{2}}X_{1}.$ 
This gives $X_{1}=q^{-\frac{1}{2}}X_{0}X_{R_p(1)}-q^{-1}X_{-1}$ which can be rewritten as $X_{1}=q^{-\frac{1}{2}}X_{V(0)}X_{R_p(1)}-q^{-1}X_{V(-1)}$ according to Lemma \ref{rank2}. Hence we have 
$X_{1}=q^{-\frac{1}{2}}\Psi(\delta_{V(0)})\Psi(\delta_{R_p(1)})-q^{-1}\Psi(\delta_{V(-1)})=
\Psi(q^{-\frac{1}{2}}\delta_{V(0)}\ast\delta_{R_p(1)}-q^{-1}\delta_{V(-1)}).$ Similarly we have
 $X_{3}X_{R_p(1)}=q^{-\frac{1}{2}}X_{2}+q^{\frac{1}{2}}X_{4},$  and using the same method we deduce that 
 $X_{2}=\Psi(q^{\frac{1}{2}}\delta_{V(3)}\ast\delta_{R_p(1)}-q\delta_{V(4)}).$ This completes the proof.
\end{proof}

\end{example}


\end{document}